\documentclass[11pt,a4paper]{amsart}
\usepackage{a4wide}
\usepackage[T1]{fontenc}
\usepackage[utf8]{inputenc}
\usepackage{lmodern}
\usepackage{amssymb}

\usepackage{url}
\usepackage[colorlinks=true,linkcolor=blue,citecolor=blue,urlcolor=blue]{hyperref}
\hypersetup{pdftitle={Ultrafilter characters and local pi-bases},pdfauthor={Jerzy Kąkol}}

\newtheorem{maintheorem}{Theorem}

\newtheorem{theorem}{Theorem}[section]
\newtheorem{proposition}[theorem]{Proposition}
\newtheorem{lemma}[theorem]{Lemma}
\newtheorem{corollary}[theorem]{Corollary}
\theoremstyle{definition}

\newtheorem{question}[theorem]{Question}
\theoremstyle{remark}
\newtheorem{remark}[theorem]{Remark}

\newcommand{\cont}{\mathfrak c}
\newcommand{\U}{\mathcal U}
\newcommand{\F}{\mathcal F}
\newcommand{\G}{\mathcal G}
\newcommand{\B}{\mathcal B}

\newcommand{\pc}{\pi\chi}
\newcommand{\pcl}{\pi\chi_l}
\newcommand{\pcc}{\pi\chi_c}
\newcommand{\pcd}{\pi\chi_d}
\DeclareMathOperator{\cf}{cf}
\DeclareMathOperator{\niso}{ni}

\title[Ultrafilter characters and local $\pi$-bases]{Ultrafilter characters and local $\pi$-bases}
\author{Jerzy K\k{a}kol}
\address{Faculty of Mathematics and Computer Science\\
Adam Mickiewicz University\\
Uniwersytetu Pozna\'nskiego 4\\
61-614 Pozna\'n, Poland}
\email{kakol@amu.edu.pl}
\date{October 1, 2026}

\begin{document}
\begin{abstract}
We completely resolve Question~5.3 of Kania's paper \cite{Kania} and establish a more general realization theorem. We show that for infinite cardinals $\mu\leqslant\nu<\kappa<\lambda$, a space with
\[
\pc(X)=\mu,\qquad\pcl(X)=\nu,\qquad
\pcc(X)=\kappa,\qquad\pcd(X)=\lambda
\]
exists if and only if $\kappa\leqslant2^\mu$. The witnesses may be zero-dimensional and Tychonoff spaces, of the least possible cardinality $\lambda$. They may have the least possible number of non-isolated points, or alternatively be crowded. The main ingredient is an exact computation of the centered local $\pi$-character of star filter spaces in terms of ultrafilter character, obtained from a finite selection lemma for filter pushforwards. We also construct a single space with countable $r$-linked local $\pi$-bases for every finite $r$, centered value $\omega_1$, and any prescribed decreasing value greater than $\omega_1$.

\end{abstract}
\maketitle

\section{Introduction}

A local $\pi$-base at a point $p$ of a space $X$ is a family of non-empty open sets such that every neighborhood of $p$ contains a member of the family. Its members need not contain $p$. Requiring intersections of these sets to remain non-empty supplies some of the coherence of a neighborhood base. Carlson \cite{Carlson} introduced linked, centered and decreasing local $\pi$-characters and applied them to cardinality bounds for Hausdorff spaces.

We write $\pc(p,X)$ for the ordinary local $\pi$-character. A family is \emph{linked} if any two members intersect, \emph{centered} if every finite subfamily has non-empty intersection, and \emph{$r$-linked}, for $2\leqslant r<\omega$, if this holds for every subfamily with at most $r$ members. The least infinite cardinalities of linked, centered and decreasing local $\pi$-bases are denoted by $\pcl(p,X)$, $\pcc(p,X)$ and $\pcd(p,X)$, respectively. A decreasing family is linearly ordered by inclusion. The last invariant is used only when such a family exists. Global invariants are suprema of pointwise values, and $\pcd(X)$ is used only when it is defined everywhere. Finite witnesses give value $\omega$. Thus
\[
\pc(X)\leqslant\pcl(X)\leqslant\pcc(X)\leqslant\pcd(X)
\]
whenever the last term is defined.

Kania \cite{Kania} answered Carlson's separation question by constructing in ZFC a zero-dimensional Tychonoff space with one non-isolated point and
\[
\pc(X)=\pcl(X)=\omega<\pcc(X)=\omega_1<\pcd(X)=\omega_2.
\]
He combines a countable star filter with a matrix filter. The star component supplies a countable linked base and obstructs small centered bases; the matrix component fixes the larger values. His argument extends to regular uncountable $\kappa<\lambda$ under $\kappa\leqslant\mathfrak p$, where $\mathfrak p$ is the pseudointersection number. Question~5.3 of \cite{Kania} asks for the corresponding realization without this additional hypothesis.

Carlson's general inequality
\begin{equation}\label{eq:carlson}
\pcc(X)\leqslant2^{\pc(X)}
\end{equation}
shows that $\kappa\leqslant\cont=2^\omega$ is necessary in that question. We prove sufficiency, remove the regularity assumption on $\kappa$, allow singular global decreasing values, and separate ordinary from linked local $\pi$-character. Let $\niso(X)$ be the number of non-isolated points of $X$.

Our first main result  is the following
\begin{maintheorem}\label{main:spectrum}
Let $\mu\leqslant\nu<\kappa<\lambda$ be infinite cardinals. The following are equivalent:
\begin{enumerate}
\item $\kappa\leqslant2^\mu$;
\item there is a space $X$ such that every point admits a decreasing local $\pi$-base and
\begin{equation}\label{eq:target}
\pc(X)=\mu,\qquad\pcl(X)=\nu,\qquad
\pcc(X)=\kappa,\qquad\pcd(X)=\lambda.
\end{equation}
\end{enumerate}
When these conditions hold, for each fixed integer $r\geqslant2$ there is a zero-dimensional Tychonoff witness of cardinality $\lambda$, with an $r$-linked local $\pi$-base of cardinality at most $\nu$ at every point, and
\[
\niso(X)=
\begin{cases}
1,&\lambda\text{ is regular},\\
\cf(\lambda),&\lambda\text{ is singular}.
\end{cases}
\]
These cardinalities are least possible among all spaces satisfying \eqref{eq:target}. The same four invariant values can also be realized by a crowded zero-dimensional Tychonoff space of cardinality $\lambda$.
\end{maintheorem}

The absence of a separation hypothesis in (2) makes the necessity assertion explicit; the construction gives the stronger separation properties stated afterward. No cardinal in this theorem is assumed regular. The case $\mu=\nu=\omega$ completely resolves Kania's question. The case $\mu<\nu$ gives four distinct values.

For a proper filter $\F$, let $\chi(\F)$ be the least infinite cardinality of a base under actual inclusion. An ultrafilter on an infinite cardinal $\mu$ is \emph{uniform} if all its members have cardinality $\mu$. Given a uniform ultrafilter $\U$ on $\mu$ and $2\leqslant r<\omega$, put
\[
E_r(\mu)=[\mu]^r,\quad L_\xi^{(r)}=\{e\in E_r(\mu):\xi\in e\},\quad
H_A^{(r)}=\{e\in E_r(\mu):e\cap A\ne\varnothing\},
\]
and define
\begin{equation}\label{eq:starfilter}
\F_\U^{(r)}=\{B\subseteq E_r(\mu):H_A^{(r)}\subseteq B
\text{ for some }A\in\U\}.
\end{equation}
For $\mu=\omega$ and $r=2$, this is Kania's star filter. Write $X(\F)$ for the associated one-point filter space.

Our next result states  the following
\begin{maintheorem}\label{main:star}
Let $\U$ be a uniform ultrafilter on an infinite cardinal $\mu$, and let $2\leqslant r<\omega$. Every proper filter $\G\supseteq\F_\U^{(r)}$ satisfies
\[
\chi(\G)\geqslant\chi(\U),\qquad
\chi(\F_\U^{(r)})=\chi(\U).
\]
At the non-isolated point $p$ of the zero-dimensional Tychonoff space $Y=X(\F_\U^{(r)})$,
\[
\pc(p,Y)=\pcl(p,Y)=\mu,\qquad \pcc(p,Y)=\chi(\U).
\]
The stars form an $r$-linked local $\pi$-base of cardinality $\mu$. The space has cardinality $\mu$, and $p$ admits no decreasing local $\pi$-base.
\end{maintheorem}

The proof uses a finite selection lemma for filter pushforwards. Posp\'i\v{s}il's theorem supplies uniform ultrafilters on $\mu$ of character $2^\mu$, see  \cite{Pospisil1939,Comfort}; we recall a proof adapted to the present notation. This strengthens Kania's star estimate and, with his matrix and projection arguments, yields Theorem~\ref{main:spectrum}. The additional linked separation uses an ultrafilter component and a separate lower-bound argument; it does not require a linked analog of the centered projection lemma.

For a fixed $r$, the witness in Theorem~\ref{main:spectrum} may depend on $r$. At the first uncountable centered value one can remove this dependence.

Finally we prove also the following
\begin{maintheorem}\label{main:simultaneous}
For every cardinal $\lambda>\omega_1$, there is a zero-dimensional Tychonoff space $X$ with
\[
\pc(X)=\pcl(X)=\omega<\pcc(X)=\omega_1<\pcd(X)=\lambda
\]
such that, for every integer $r\geqslant2$, every point of this same space has a countable $r$-linked local $\pi$-base. One may have $|X|=\lambda$ and the optimal number of non-isolated points from Theorem~\ref{main:spectrum}, or instead choose $X$ crowded with the same cardinality.
\end{maintheorem}
\section{Filter spaces and decreasing bases}

A proper filter $\F$ on a set $S$ is \emph{free} if $\bigcap\F=\varnothing$. A family $\B$ of non-empty subsets of $S$ is \emph{coinitial in $\F$} if every member of $\F$ contains some member of $\B$; we do not require $\B\subseteq\F$. The space $X(\F)=S\cup\{p\}$ has all points of $S$ isolated and basic neighborhoods $\{p\}\cup A$, $A\in\F$, at $p$.

Freeness implies that $S\setminus a\in\F$ for every finite $a\subseteq S$. Consequently, $S$ and all filter members are infinite, and every non-empty open subset of $X(\F)$ meets $S$. The basic neighborhoods of $p$ and the singleton subsets of $S$ are clopen. Thus $X(\F)$ is zero-dimensional and Tychonoff, with $p$ its unique non-isolated point.

We shall use the following trace observation from \cite[Lemmas 2.3 and 2.4]{Kania}.

\begin{lemma}\label{lem:traces}
Let $\F$ be a free filter on $S$. A family $\B$ of non-empty open subsets of $X(\F)$ is a local $\pi$-base at $p$ if and only if its traces $\{B\cap S:B\in\B\}$ are coinitial in $\F$. Taking traces preserves linkedness, centeredness, every finite degree of linkedness, and inclusion chains. If $\B$ is a centered local $\pi$-base, every finite intersection of its traces is infinite.
\end{lemma}
\begin{proof}
The first assertion follows directly from the basic neighborhoods $\{p\}\cup A$. Every non-empty finite intersection of open sets is open and therefore meets $S$, which proves preservation of the intersection properties; inclusion is also preserved. For the last assertion, consider a non-empty finite subfamily $\B_0\subseteq\B$. If each member contains $p$, its trace belongs to $\F$, so their intersection is infinite. Otherwise some member omits $p$. If the intersection of the traces were a finite set $a$, locality would provide $B_1\in\B$ contained in $\{p\}\cup(S\setminus a)$. Then the intersection of $\B_0\cup\{B_1\}$ would be empty, contradicting centeredness. The empty intersection is $S$, which is infinite.
\end{proof}
Next two propositions will be used in the proof of Theorem \ref{main:star}.
\begin{proposition}\label{prop:extensions}
Let $\F$ be a free filter on $S$. Then
\begin{equation}\label{eq:extensions}
\pcc(p,X(\F))=
\min\{\chi(\G):\G\text{ is a proper filter on }S,\ \F\subseteq\G\}.
\end{equation}
\end{proposition}
\begin{proof}
The traces of a centered local $\pi$-base generate a proper filter $\G$ by Lemma~\ref{lem:traces}. Coinitiality gives $\F\subseteq\G$. Closing a family of cardinality at most an infinite cardinal $\tau$ under finite intersections gives a base of cardinality at most $\tau$, so this proves one inequality. Conversely, a base for any proper filter $\G\supseteq\F$ is a centered family of non-empty open subsets of the discrete part $S$ and is coinitial in $\F$. It is therefore a centered local $\pi$-base at $p$.
\end{proof}

For a map $f:S\to T$, the pushforward of a proper filter is
\[
f_*(\F)=\{A\subseteq T:f^{-1}(A)\in\F\}.
\]
If $\B$ is a base for $\F$, then $\{f[B]:B\in\B\}$ is a base for $f_*(\F)$. In particular,
\begin{equation}\label{eq:pushcharacter}
\chi(f_*(\F))\leqslant\chi(\F).
\end{equation}
A set $a\subseteq S$ is \emph{$\F$-positive} if it meets every member of $\F$. In this case the filter $\F_a$ generated by $\F\cup\{a\}$ is proper, with base $\{B\cap a:B\in\B\}$, and
\begin{equation}\label{eq:restrictcharacter}
\chi(\F_a)\leqslant\chi(\F).
\end{equation}

The next fact, including the equivalence of the conventions for decreasing bases, is Kania's chain-refinement theorem \cite[Theorem B and Corollary 2.2]{Kania}.

\begin{proposition}\label{prop:chains}
Every inclusion-chain local $\pi$-base at $p\in X$ has a coinitial subfamily which can be enumerated as a strictly decreasing transfinite sequence, of cardinality at most $\max\{\omega,|X|\}$ and no greater than the original family. Consequently, when $\pcd(p,X)$ is defined, it is either $\omega$ or a regular uncountable cardinal, and
\[
\pcd(p,X)\leqslant\max\{\omega,|X|\}.
\]
In particular, a countable space with $\pcc(p,X)>\omega$ admits no decreasing local $\pi$-base at $p$.
\end{proposition}
\begin{proof}
From a non-empty chain choose successively a member strictly contained in every previously chosen member, stopping as soon as the selected family is coinitial. The chain property makes the next choice possible whenever coinitiality fails. If $X$ is infinite, this recursion must stop before $|X|^+$: otherwise a point chosen from each successive difference would give $|X|^+$ distinct points of $X$. The finite case is immediate.

For regularity, suppose $\mu=\pcd(p,X)>\omega$. Refine a witnessing chain to a strictly decreasing sequence indexed by an ordinal $\theta$ with $|\theta|\leqslant\mu$. Minimality gives $|\theta|=\mu$. The ordinal $\theta$ is not a successor, since a last member would alone be a local $\pi$-base. Passing to a cofinal subsequence gives $\mu\leqslant\cf(\theta)\leqslant|\theta|=\mu$, so $\mu$ is regular. Finally, a chain of non-empty sets is centered.
\end{proof}

The regularity assertion is pointwise. A global decreasing local $\pi$-character, being a supremum, can be singular.

\section{Ultrafilters and finite families of maps}
First we prove the following lemma which will be used for the proof of Theorem \ref{main:star}.
\begin{lemma}\label{lem:finite}
Let $\F$ be a proper filter on $S$, let $\U$ be an ultrafilter on $T$, and let $f_i:S\to T$ for $i<r$, where $1\leqslant r<\omega$. Suppose
\begin{equation}\label{eq:finiteunion}
\bigcup_{i<r}f_i^{-1}(A)\in\F\qquad(A\in\U).
\end{equation}
Then there are an $\F$-positive set $a\subseteq S$ and an index $i<r$ such that
\[
(f_i)_*(\F_a)=\U.
\]
In particular, $\chi(\U)\leqslant\chi(\F)$.
\end{lemma}
\begin{proof}
Choose $J\subseteq r$ of largest cardinality for which there are sets $A_j\in\U$, $j\in J$, such that
\[
a=S\setminus\bigcup_{j\in J}f_j^{-1}(A_j)
\]
is $\F$-positive. The empty choice is admissible. We have $J\ne r$: otherwise $A=\bigcap_{j<r}A_j\in\U$, and the filter member $\bigcup_{j<r}f_j^{-1}(A)$ would be disjoint from $a$.

Fix $i\notin J$. For every $A\in\U$, maximality says that $a\setminus f_i^{-1}(A)$ is not $\F$-positive. Thus some $B\in\F$ satisfies $B\cap a\subseteq f_i^{-1}(A)$, so $f_i^{-1}(A)\in\F_a$. The proper filter $(f_i)_*(\F_a)$ contains the ultrafilter $\U$ and therefore equals it. Equations~\eqref{eq:pushcharacter} and \eqref{eq:restrictcharacter} give the character bound.
\end{proof}

We shall use the following classical ultrafilter facts. Posp\'i\v{s}il's 1937 paper \cite{Pospisil1937} belongs to the independent-family construction of large families of ultrafilters. The maximal-character result used here is attributed to his 1939 paper \cite{Pospisil1939}; see \cite[Theorem 2.7]{Comfort} and the countable version in \cite{GuzmanHrusak}.

\begin{proposition}\label{prop:uniform}
For every infinite cardinal $\mu$ there is a uniform ultrafilter $\U$ on $\mu$ with $\chi(\U)=2^\mu$. Every uniform ultrafilter on $\mu$ has character strictly greater than $\mu$.
\end{proposition}
\begin{proof}
We recall the independent-family argument for the first assertion. Let
\[
\Omega=\{(s,H):s\in[\mu]^{<\omega},\ H\subseteq\mathcal P(s)\}.
\]
This set has cardinality $\mu$. For each $x\subseteq\mu$, put
\[
I_x=\{(s,H)\in\Omega:x\cap s\in H\}.
\]
The family $\mathcal I=\{I_x:x\subseteq\mu\}$ has cardinality $2^\mu$, and every finite Boolean combination of distinct members has cardinality $\mu$. Indeed, choose a finite $s_0$ distinguishing the finitely many sets $x$ involved. For each finite $s\supseteq s_0$, choose $H$ to realize the prescribed membership pattern. There are $\mu$ such sets $s$.

Generate a filter by all members of $\mathcal I$, all complements of subsets of $\Omega$ of cardinality less than $\mu$, and all sets
\[
\Omega\setminus\bigcap\mathcal J
\qquad(\mathcal J\in[\mathcal I]^\omega).
\]
Every finite intersection of generators has cardinality $\mu$. To verify this, for each of the finitely many countably infinite families $\mathcal J$ choose a member outside the finite list of positive generators. The Boolean combination making the positive generators true and the chosen members false has cardinality $\mu$; removing finitely many sets of cardinality less than $\mu$ preserves that cardinality. The Kuratowski--Zorn lemma extends this filter to an ultrafilter $\U$, which is uniform.

If $\B$ were a base of $\U$ of cardinality $\tau<2^\mu$, choose $B_I\in\B$ with $B_I\subseteq I$ for each $I\in\mathcal I$. Some base member must be chosen for infinitely many distinct $I$: otherwise $2^\mu\leqslant\tau\cdot\omega<2^\mu$. Choose a countably infinite family $\mathcal J$ with that common base member $B$. Then $B\subseteq\bigcap\mathcal J$ while $\Omega\setminus\bigcap\mathcal J\in\U$, a contradiction. Identifying $\Omega$ with $\mu$ proves the first assertion.

For the second, suppose a uniform ultrafilter had a base $\{B_\alpha:\alpha<\mu\}$, repeating members if necessary. Recursively choose distinct points $x_\alpha,y_\alpha\in B_\alpha$ avoiding every earlier choice. At each stage fewer than $\mu$ points are excluded, even when $\mu$ is singular. The disjoint sets $\{x_\alpha:\alpha<\mu\}$ and $\{y_\alpha:\alpha<\mu\}$ both meet every base member. Both would therefore belong to the ultrafilter, which is impossible.
\end{proof}

In the countable case, the last argument can also be expressed by first making a hypothetical countable base decreasing by finite intersections. A sequence of distinct points chosen diagonally from that base has even and odd subsequences both positive for the ultrafilter, giving the same contradiction.

\section{Star spaces and the ultrafilter character spectrum}

Let $\U$ be uniform on $\mu$. The sets in \eqref{eq:starfilter} generate a proper filter, since
\[
\varnothing\ne H_{A\cap B}^{(r)}\subseteq H_A^{(r)}\cap H_B^{(r)}
\qquad(A,B\in\U).
\]
If $a\subseteq E_r(\mu)$ has cardinality less than $\mu$, then $|\bigcup a|<\mu$ and uniformity gives
\[
A=\mu\setminus\bigcup a\in\U,\qquad H_A^{(r)}\cap a=\varnothing.
\]
In particular the filter is free. The stars are coinitial because $L_\xi^{(r)}\subseteq H_A^{(r)}$ whenever $\xi\in A$. Any at most $r$ stars intersect: extend their set of centers to an $r$-element subset of $\mu$.

\begin{proof}[Proof of Theorem~\ref{main:star}]
For $e\in E_r(\mu)$, let $f_i(e)$ be its $i$-th element in increasing order, $i<r$. Then
\[
H_A^{(r)}=\bigcup_{i<r}f_i^{-1}(A).
\]
If $\G\supseteq\F_\U^{(r)}$ is proper, Lemma~\ref{lem:finite} gives $\chi(\U)\leqslant\chi(\G)$. If $\mathcal A$ is a base for $\U$, then $\{H_A^{(r)}:A\in\mathcal A\}$ is a base for $\F_\U^{(r)}$. This proves the character assertions, and Proposition~\ref{prop:extensions} gives $\pcc(p,Y)=\chi(\U)$.

The stars give $\pcl(p,Y)\leqslant\mu$. Suppose $\mathcal P$ is a family of non-empty open sets with $|\mathcal P|<\mu$. Choose an edge $e_P\in P\cap E_r(\mu)$ for each $P\in\mathcal P$ and put $C=\bigcup_{P\in\mathcal P}e_P$. Since $|C|<\mu$, the set $A=\mu\setminus C$ belongs to $\U$. The neighborhood $\{p\}\cup H_A^{(r)}$ contains none of the selected edges, and therefore contains no member of $\mathcal P$. Thus $\pc(p,Y)\geqslant\mu$.

The space has cardinality $\mu$, whereas its centered local $\pi$-character is $\chi(\U)>\mu$ by Proposition~\ref{prop:uniform}. Proposition~\ref{prop:chains} therefore excludes a decreasing local $\pi$-base at $p$.
\end{proof}

\begin{remark}\label{rem:original}
For $\mu=\omega$ and $r=2$, Kania defines his filter by requiring $\{n:L_n^{(2)}\subseteq B\}\in\U$. This is equivalent to \eqref{eq:starfilter}: in one direction take this set as $A$, and in the other use $A\subseteq\{n:L_n^{(2)}\subseteq B\}$. Theorem~\ref{main:star} computes the invariant of the original star space itself.
\end{remark}

Write
\[
\operatorname{Sp}_\chi(\omega)=
\{\chi(\U):\U\text{ is a free ultrafilter on }\omega\},
\qquad \mathfrak u=\min\operatorname{Sp}_\chi(\omega).
\]
Here $\mathfrak u$ is the ultrafilter number. The actual-inclusion and almost-inclusion definitions give the same infinite character on $\omega$, by adjoining the countably many finite deletions.

\begin{corollary}\label{cor:spectrum}
For each fixed $r\geqslant2$, the possible centered local $\pi$-characters at the non-isolated points of the countable star spaces $X(\F_\U^{(r)})$ are exactly $\operatorname{Sp}_\chi(\omega)$. Their minimum is $\mathfrak u$, and their maximum is $\cont$. In particular, Kania's lower estimate $\mathfrak p$ improves to $\mathfrak u$ for every defining free ultrafilter.
\end{corollary}
\begin{proof}
The exact spectrum and its extrema follow from Theorem~\ref{main:star}, the definition of $\mathfrak u$, and Proposition~\ref{prop:uniform}. For completeness, $\mathfrak p\leqslant\mathfrak u$: a free ultrafilter base has the strong finite-intersection property and no infinite pseudointersection. If $P$ were such a pseudointersection, split it into two infinite pieces. One piece has complement in the ultrafilter, contradicting almost containment in that complement.
\end{proof}

\begin{corollary}\label{cor:countable}
For each infinite $\mu$ and fixed finite $r\geqslant2$, there is a zero-dimensional Tychonoff space $Y$ of cardinality $\mu$, with one non-isolated point, such that
\[
\pc(Y)=\pcl(Y)=\mu,\qquad\pcc(Y)=2^\mu.
\]
Every point has an $r$-linked local $\pi$-base of cardinality at most $\mu$, and the non-isolated point has no decreasing local $\pi$-base.
\end{corollary}
\begin{proof}
Take a uniform ultrafilter on $\mu$ of character $2^\mu$ and apply Theorem~\ref{main:star}.
\end{proof}
\section{The matrix component and finite sums}

We recall Kania's matrix construction \cite[Section 4 and Remark 5.2]{Kania}. The proof below records that regularity of the smaller cardinal is unnecessary.

Let $\omega<\kappa<\lambda$, where $\lambda$ is regular, and set
\[
S=\kappa\times\lambda,\qquad
d_{\alpha,F}=(\kappa\times\alpha)\cup(F\times\lambda)
\quad(\alpha<\lambda,\ F\in[\kappa]^{<\omega}).
\]
Put $D=\{d_{\alpha,F}:\alpha<\lambda,\ F\in[\kappa]^{<\omega}\}$. Each element of $D$ is both a point of the discrete space used below and a subset of $S$. For finite $e\subseteq S$, let
\[
G_e=\{d\in D:e\subseteq d\},\qquad
\G_{\kappa,\lambda}=\{A\subseteq D:G_e\subseteq A
\text{ for some }e\in[S]^{<\omega}\}.
\]
The representation $d_{\alpha,F}$ is unique: its full columns determine $F$, and any column outside $F$ determines $\alpha$. Hence $|D|=\lambda$.

\begin{proposition}\label{prop:matrix}
The family $\G_{\kappa,\lambda}$ is a free filter. At the non-isolated point $q$ of $Z=X(\G_{\kappa,\lambda})$,
\[
\pc(q,Z)=\pcl(q,Z)=\pcc(q,Z)=\kappa,
\qquad \pcd(q,Z)=\lambda.
\]
\end{proposition}
\begin{proof}
Every $G_e$ is non-empty: choose $\alpha$ above the second coordinates appearing in $e$. Also $G_e\cap G_f=G_{e\cup f}$. To omit finitely many $d_{\alpha_j,F_j}$, choose $\xi\in\kappa\setminus\bigcup_jF_j$ and $\beta\geqslant\alpha_j$ for every $j$. Then $G_{\{(\xi,\beta)\}}$ omits those points. This proves freeness.

For finite $F\subseteq\kappa$, let
\[
B_F=\{d_{\alpha,H}:F\subseteq H\}.
\]
These sets form a centered family of cardinality $\kappa$, and they are coinitial in $\G_{\kappa,\lambda}$: take $F$ to contain all first coordinates of a given finite $e$. Thus $\pcc(q,Z)\leqslant\kappa$.

Suppose $\mathcal P$ is a family of non-empty open sets with $|\mathcal P|<\kappa$. Choose $d_{\alpha_P,F_P}\in P\cap D$ for each $P\in\mathcal P$. Since $\bigcup_{P\in\mathcal P}F_P$ has cardinality less than $\kappa$, choose $\xi$ outside this union. Regularity of $\lambda$ gives $\beta<\lambda$ with $\beta\geqslant\alpha_P$ for all $P\in\mathcal P$. None of the chosen points belongs to $G_{\{(\xi,\beta)\}}$, so no member of $\mathcal P$ is contained in $\{q\}\cup G_{\{(\xi,\beta)\}}$. Hence $\pc(q,Z)\geqslant\kappa$.

The sets
\[
C_\alpha=\{d_{\beta,F}:\beta\geqslant\alpha\}\qquad(\alpha<\lambda)
\]
form a decreasing coinitial family, so $\pcd(q,Z)\leqslant\lambda$. To prove the reverse inequality, let $\mathcal P$ be an inclusion-chain local $\pi$-base with $|\mathcal P|<\lambda$. For $P\in\mathcal P$, put
\[
T_P=P\cap D,\qquad K_P=\bigcap_{d\in T_P}d\subseteq S,
\qquad F_P^*=\{\xi<\kappa:\{\xi\}\times\lambda\subseteq K_P\}.
\]
All $T_P$ are non-empty. Each $F_P^*$ is finite, and the family of these sets is a chain, since $P\mapsto K_P$ reverses inclusion. A chain of finite sets has only countably many distinct members. Thus $\bigcup_{P\in\mathcal P}F_P^*$ is countable; since $\kappa>\omega$, choose $\xi$ outside it.

For each $P$, choose $d_{\alpha_P,F_P}\in T_P$ with $\xi\notin F_P$. Then
\[
K_P\cap(\{\xi\}\times\lambda)\subseteq\{\xi\}\times\alpha_P.
\]
Regularity of $\lambda$ gives $\beta=\sup_{P\in\mathcal P}\alpha_P<\lambda$, so $(\xi,\beta)$ belongs to no $K_P$. However, locality applied to $\{q\}\cup G_{\{s\}}$ for each $s\in S$ implies $S=\bigcup_{P\in\mathcal P}K_P$. This contradiction proves the claim.
\end{proof}

Let $\F_0$ and $\F_1$ be free filters on disjoint sets $S_0$ and $S_1$. Their sum is the free filter
\[
\F_0\oplus\F_1=
\{A\subseteq S_0\mathbin{\dot\cup}S_1:A\cap S_i\in\F_i\ (i<2)\}.
\]
Here $S_0\mathbin{\dot\cup}S_1$ denotes the disjoint union: it is the ordinary union $S_0\cup S_1$, with the dot emphasizing that $S_0\cap S_1=\varnothing$.
The centered and decreasing assertions below use Kania's projection argument
\cite[Lemma 2.7]{Kania}; we also record the ordinary counterpart.

\begin{lemma}\label{lem:sum}
For a sum of two free filters, ordinary local $\pi$-character and centered local $\pi$-character are the minima of the respective component values. A decreasing local $\pi$-base exists in the sum exactly when it exists in a component, and its least infinite cardinality is the minimum over the components in which it exists. These conclusions extend to finite sums.
\end{lemma}
\begin{proof}
Let $\B$ be a local $\pi$-base at the non-isolated point of $X(\F_0\oplus\F_1)$. For ordinary character, take the non-empty coordinate traces
\[
\mathcal P_i=\{B\cap S_i:B\in\B,\ B\cap S_i\ne\varnothing\}.
\]
If neither family were coinitial, choose $A_i\in\F_i$ containing no member of $\mathcal P_i$. Locality would give $B\subseteq\{p\}\cup A_0\cup A_1$. At least one trace is non-empty, contradicting the choice of $A_i$. Thus one coordinate supplies an ordinary local $\pi$-base of cardinality at most $|\B|$.

If $\B$ is centered, its traces on $S_0\mathbin{\dot\cup}S_1$ have the finite-intersection property. If both coordinate families failed that property, finitely many members witnessing failure in each coordinate would together contradict it. Hence one full coordinate family is centered and, by locality, coinitial in the corresponding filter. Projection also preserves chains. Conversely, a coinitial family in one component remains coinitial in the sum and retains each of the stated properties. Induction gives the finite-sum assertions.
\end{proof}

Linked families need not admit such a coordinate projection. For an ultrafilter component a different argument is available.

\begin{lemma}\label{lem:ultra-linked}
Let $\U$ be a free ultrafilter on $S_0$ and $\F$ a free filter on a disjoint set $S_1$. Write $a=\pc(p_1,X(\F))$ and $b=\pcl(p_1,X(\F))$. At the non-isolated point $p$ of $X(\U\oplus\F)$,
\[
\min\{\chi(\U),a\}\leqslant\pcl(p,X(\U\oplus\F))
\leqslant\min\{\chi(\U),b\}.
\]
\end{lemma}
\begin{proof}
The upper bound uses a base of $\U$ or a linked local $\pi$-base in the second component. For the lower bound, suppose $\B$ is a linked local $\pi$-base of cardinality $\tau<\min\{\chi(\U),a\}$. Replace its members by their traces on $S_0\mathbin{\dot\cup}S_1$, using Lemma~\ref{lem:traces}. Since the non-empty $S_1$-traces have cardinality less than $a$, there is $A_1\in\F$ containing none of them. For each $A_0\in\U$, locality gives $B\in\B$ with $B\subseteq A_0\cup A_1$. Its $S_1$-trace is empty, so $B\subseteq A_0$.

It follows that $\B_0=\{B\in\B:B\subseteq S_0\}$ is linked and coinitial in $\U$. Every member of $\B_0$ belongs to $\U$: otherwise its complement belongs to $\U$, and coinitiality supplies a disjoint member of $\B_0$. Hence $\B_0$ is an actual base for $\U$, contradicting $\tau<\chi(\U)$.
\end{proof}

\begin{lemma}\label{lem:ultra-space}
If $\U$ is uniform on an infinite cardinal $\mu$, then at the non-isolated point of $X(\U)$ the ordinary local $\pi$-character is $\mu$, the linked and centered local $\pi$-characters are both $\chi(\U)$, and no decreasing local $\pi$-base exists.
\end{lemma}
\begin{proof}
Singletons give the ordinary upper bound. Choosing one point from each of fewer than $\mu$ proposed members and using the complement of those points in $\U$ proves the lower bound. A linked coinitial family consists of $\U$-positive sets and therefore of members of $\U$, so its cardinality is at least $\chi(\U)$. A base for $\U$ gives the matching centered upper bound. Finally, $\chi(\U)>\mu=|X(\U)|$, and Proposition~\ref{prop:chains} applies.
\end{proof}

\section{Realization and optimality}

Given a local $\pi$-base $\B$ at $p$, put
\[
V_U=\bigcup\{B\in\B:B\subseteq U\}
\]
for each open neighborhood $U$ of $p$. These sets are non-empty, open and contained in $U$, and
\[
\varnothing\ne V_{U_1\cap\cdots\cap U_n}
\subseteq V_{U_1}\cap\cdots\cap V_{U_n}.
\]
There are at most $2^{|\B|}$ possible unions. This proves Carlson's inequality \eqref{eq:carlson} for all spaces \cite[Proposition 2.8]{Carlson}, and establishes the necessity in Theorem~\ref{main:spectrum}.

\begin{proof}[Proof of Theorem~\ref{main:spectrum}, apart from the crowded assertion]
Assume $\kappa\leqslant2^\mu$, and fix $r\geqslant2$. Suppose first that $\lambda$ is regular. Choose a uniform ultrafilter $\mathcal V$ on $\nu$ with $\chi(\mathcal V)=2^\nu$, and form on disjoint ground sets
\[
\F=\F_{\mathcal V}^{(r)}\oplus\G_{\kappa,\lambda}.
\]
At the non-isolated point of $X(\F)$, Theorem~\ref{main:star}, Proposition~\ref{prop:matrix} and Lemma~\ref{lem:sum} give
\[
\pc=\pcl=\nu,\qquad\pcc=\min\{2^\nu,\kappa\}=\kappa,
\qquad\pcd=\lambda.
\]
Here the linked upper bound comes from the stars and the ordinary lower bound from the sum lemma. The stars also give an $r$-linked local $\pi$-base of cardinality $\nu$.

If $\mu=\nu$, take $X=X(\F)$. If $\mu<\nu$, choose a uniform ultrafilter $\U$ on $\mu$ with $\chi(\U)=2^\mu$ and set
\[
X=X(\U\oplus\F).
\]
Lemma~\ref{lem:sum} gives ordinary value $\mu$, centered value $\min\{2^\mu,\kappa\}=\kappa$, and decreasing value $\lambda$, since the ultrafilter component has no decreasing base. Lemma~\ref{lem:ultra-linked} gives linked value $\nu$, using $\nu<\kappa\leqslant2^\mu$. The stars in the $\mathcal V$-component retain their $r$-linked coinitial property. In both cases $X$ is zero-dimensional and Tychonoff, has cardinality $\lambda$, and has exactly one non-isolated point. All other points contribute only $\omega$ to the global invariants.

For singular $\lambda$, put $\delta=\cf(\lambda)$ and choose regular cardinals $\lambda_i$, $i<\delta$, such that
\[
\kappa<\lambda_i<\lambda,\qquad\sup_{i<\delta}\lambda_i=\lambda.
\]
They may be chosen as successor cardinals above a cofinal sequence; distinctness is not required. Let $X_i$ be the regular-case example with values $\mu,\nu,\kappa,\lambda_i$, and take the topological sum $X=\coprod_{i<\delta}X_i$. Local invariants are unchanged on passing to an open subspace containing the point: restrict a local $\pi$-base to members contained in that subspace, retaining locality and the intersection or inclusion properties; conversely, its open sets are open in the whole space. Therefore the global values are $\mu,\nu,\kappa,\lambda$, every point retains the required $r$-linked base, and
\[
|X|=\sum_{i<\delta}\lambda_i=\lambda,\qquad\niso(X)=\delta.
\]

Proposition~\ref{prop:chains} gives $\pcd(X)\leqslant\max\{\omega,|X|\}$ without any separation assumption, so $|X|\geqslant\lambda$ is necessary. At least one point is non-isolated. If $\lambda$ is singular, every pointwise decreasing local $\pi$-character is strictly below $\lambda$. Fewer than $\cf(\lambda)$ non-isolated points would give a supremum below $\lambda$, since isolated points contribute only $\omega$. This proves the claimed optimality. The crowded assertion follows from Proposition~\ref{prop:crowded} below.
\end{proof}

\begin{corollary}\label{cor:onepoint}
For infinite $\mu\leqslant\nu<\kappa<\lambda$, a Hausdorff space with exactly one non-isolated point and the values in \eqref{eq:target} exists if and only if $\kappa\leqslant2^\mu$ and $\lambda$ is regular.
\end{corollary}
\begin{proof}
Use Theorem~\ref{main:spectrum}, Carlson's bound, and pointwise regularity of the decreasing invariant.
\end{proof}

\begin{corollary}\label{cor:ch}
Over ZFC, failure of the continuum hypothesis is equivalent to each of the following assertions:
\begin{enumerate}
\item there is a space with $\pc=\pcl=\omega$, $\pcc=\omega_2$, and $\pcd=\omega_3$;
\item there is a space with $\pc=\omega$, $\pcl=\omega_1$, $\pcc=\omega_2$, and $\pcd=\omega_3$.
\end{enumerate}
In either case the witness can be zero-dimensional and Tychonoff, of cardinality $\omega_3$, with exactly one non-isolated point, or instead crowded.
\end{corollary}
\begin{proof}
Both conditions in Theorem~\ref{main:spectrum} reduce to $\omega_2\leqslant\cont$.
\end{proof}

\section{Simultaneous finite linkedness}

Fix a free ultrafilter $\U$ on $\omega$. Use disjoint copies of $E_r=[\omega]^r$, $r\geqslant2$, and put
\[
S=\bigsqcup_{r\geqslant2}E_r,\qquad
\Sigma=\{A\subseteq S:A\cap E_r\in\F_\U^{(r)}\text{ for all }r\geqslant2\}.
\]
This is a free filter on a countable set. For every fixed $r$, the stars supported on $E_r$ give a countable $r$-linked family coinitial in $\Sigma$.

\begin{proposition}\label{prop:simultaneous}
There is no countably based proper filter extending $\Sigma$. Consequently, at the non-isolated point of $X(\Sigma)$,
\[
\pc=\pcl=\omega<\pcc,
\]
and no decreasing local $\pi$-base exists.
\end{proposition}
\begin{proof}
Suppose $\G\supseteq\Sigma$ is proper with a countable base. If some $E_r$ is $\G$-positive, the trace on $E_r$ of $\G_{E_r}$ is a countably based proper filter extending $\F_\U^{(r)}$. Indeed, lift a member of that coordinate filter by taking full sets in the other coordinates. This contradicts Theorem~\ref{main:star}.

Otherwise $S\setminus E_r\in\G$ for every $r$. Every member of $\G$ then meets infinitely many coordinates, since the complements of any finitely many coordinates also belong to $\G$. Enumerate a base as $\{B_k:k<\omega\}$. Choose distinct $r_k\geqslant2$ and edges $e_k\in B_k\cap E_{r_k}$. Define $A\subseteq S$ by
\[
A\cap E_{r_k}=H_{\omega\setminus e_k}^{(r_k)}\quad(k<\omega),
\]
and let its other coordinate traces be the full coordinates. Each $e_k$ is finite, so $A\in\Sigma$. But $e_k\in B_k\setminus A$ for every $k$, contradicting the base property. Proposition~\ref{prop:extensions} now gives $\pcc>\omega$, and Proposition~\ref{prop:chains} excludes decreasing bases.
\end{proof}

\begin{proof}[Proof of Theorem~\ref{main:simultaneous}]
For regular $\lambda>\omega_1$, form $X(\Sigma\oplus\G_{\omega_1,\lambda})$. The coordinate stars give a countable $r$-linked local $\pi$-base for every finite $r$. The centered value is exactly $\omega_1$ by Lemma~\ref{lem:sum} and Proposition~\ref{prop:simultaneous}, and the decreasing value is $\lambda$. The space has size $\lambda$ and one non-isolated point. For singular $\lambda$, take a topological sum over $\cf(\lambda)$ regular decreasing values cofinal in $\lambda$, as in Theorem~\ref{main:spectrum}. The same optimality argument applies. Proposition~\ref{prop:crowded} supplies the crowded versions.
\end{proof}

\section{Crowded realizations and further questions}

The isolated points in the preceding constructions can be replaced without changing the local invariants at the distinguished point. This gives crowded examples without making any claim of countable compactness or compactness.

\begin{proposition}\label{prop:crowded}
Let $\F$ be a free filter on an infinite set $S$. There is a crowded zero-dimensional Tychonoff space $Y$ of cardinality $|S|$, with a point $p$, such that all four local $\pi$-characters at $p$ agree with those at the non-isolated point of $X(\F)$, including existence of the decreasing invariant. The existence of an $r$-linked local $\pi$-base of any given infinite size bound is preserved for each finite $r$. At every other point of $Y$, all four invariants are $\omega$.
\end{proposition}
\begin{proof}
Replace each $s\in S$ by a disjoint copy $\mathbb Q_s$ of the rationals, and adjoin $p$ with basic neighborhoods
\[
\{p\}\cup\bigcup_{s\in A}\mathbb Q_s\qquad(A\in\F).
\]
Each fiber is clopen, since $S\setminus\{s\}\in\F$, and each displayed neighborhood is clopen. The clopen bases of the rational fibers therefore give a clopen base for $Y$. The space is Hausdorff, hence zero-dimensional Tychonoff; it is crowded and has cardinality $|S|$.

For a non-empty open set $B\subseteq Y$, let
\[
\operatorname{supp}(B)=\{s\in S:B\cap\mathbb Q_s\ne\varnothing\}.
\]
Every non-empty open set meets a fiber. The support map sends a local $\pi$-base at $p$ to a coinitial family in $\F$, preserves inclusions, and preserves each finite intersection property under consideration: a non-empty finite intersection is open and meets some fiber. Conversely, take traces of a local $\pi$-base in $X(\F)$ on $S$ and replace each trace $T$ by $\bigcup_{s\in T}\mathbb Q_s$. Lemma~\ref{lem:traces} and the definition of the neighborhoods show that locality and all the required properties are preserved. These two operations prove the equalities. Points inside the open rational fibers are first countable and have all four invariants equal to $\omega$.
\end{proof}

Apply this replacement to each one-point component in the proofs of Theorems~\ref{main:spectrum} and~\ref{main:simultaneous}. This proves their crowded assertions. In particular, the existence results are not confined to spaces whose non-isolated part is small, although the crowded realizations do not have the optimal number of non-isolated points.

The finite-sum argument and the simultaneous construction leave two distinct questions which are not decided by the present methods.

\begin{question}
Suppose $\omega_1<\kappa\leqslant\cont$ and $\lambda>\kappa$ is regular. Is there a Hausdorff space with $\pc=\pcl=\omega$, $\pcc=\kappa$, and $\pcd=\lambda$, such that every point has a countable $r$-linked local $\pi$-base for every finite $r$? Can this be achieved with one non-isolated point?
\end{question}

\begin{question}
Which of the tuples in Theorem~\ref{main:spectrum} can be realized by countably compact Hausdorff spaces? Which can be realized by compact Hausdorff spaces?
\end{question}

\section*{Acknowledgments}
The author thanks Tomasz Kania for explaining the problem and for sharing an early version of his paper \emph{Three local $\pi$-characters}.

\bigskip\bigskip

\textbf{Data Availability}:  Data sharing not applicable to this article as no datasets were generated or analyzed during
the current study.

\textbf{Declarations}

\bigskip\bigskip

\textbf{Statements and Declarations}:  The author declare that no support was received during the preparation of

\textbf{Conflict of interest}: There is no conflict of interest.

\end{document}